\documentclass[11pt, a4paper]{amsart}

\usepackage{amsmath,amssymb,amsthm}
\usepackage{hyperref}
\usepackage[dvipsnames,usenames]{color}
\usepackage[latin1]{inputenc}
\usepackage{graphicx}
\usepackage{psfrag}
\usepackage{color}
\usepackage{soul}

\newtheorem{guia}{}

\newtheorem{teorema}[guia]{Theorem}

\newtheorem{lema}[guia]{Lemma}

\newcommand{\al}{\alpha}

\newcommand{\De}{\Delta}
\newcommand{\de}{\delta}
\newcommand{\ds}{\displaystyle}
\newcommand{\e}{\varepsilon}

\newcommand{\g}{\gamma}
\newcommand{\G}{\Gamma}

\newcommand{\la}{\lambda}
\newcommand{\La}{\Lambda}
\newcommand{\N}{\mathbb N}

\newcommand{\Om}{\Omega}
\newcommand{\Omb}{\overline{\Omega}}
\newcommand{\p}{\partial}

\newcommand{\R}{\mathbb R}

\begin{document}

\title[Liouville theorem for indefinite fractional diffusion]{A Liouville theorem for indefinite 
fractional diffusion equations and\\ its application to existence of solutions}

\author[B. Barrios, L. Del Pezzo, J. Garc\'{\i}a-Meli\'{a}n and A. Quaas]
{B. Barrios, L. del Pezzo, J. Garc\'{\i}a-Meli\'{a}n\\ and A. Quaas}

\date{}

\address{B. Barrios \hfill\break\indent
Departamento de An\'{a}lisis Matem\'{a}tico, Universidad de La Laguna
\hfill \break \indent C/. Astrof\'{\i}sico Francisco S\'{a}nchez s/n, 38200 -- La Laguna, SPAIN}
\email{{\tt bbarrios@ull.es}}

\address{L. Del Pezzo \hfill\break\indent
CONICET  \hfill\break\indent
Departamento de Matem\'atica y Estad\'istica
\hfill\break\indent Universidad Torcuato Di Tella
\hfill\break\indent Av. Figueroa Alcorta 7350 (C1428BCW)
\hfill\break\indent C. A. de Buenos Aires,
ARGENTINA. }
\email{{\tt ldelpezzo@utdt.edu}}

\address{J. Garc\'{\i}a-Meli\'{a}n \hfill\break\indent
Departamento de An\'{a}lisis Matem\'{a}tico, Universidad de La Laguna
\hfill \break \indent C/. Astrof\'{\i}sico Francisco S\'{a}nchez s/n, 38200 -- La Laguna, SPAIN
\hfill\break\indent
{\rm and} \hfill\break
\indent Instituto Universitario de Estudios Avanzados (IUdEA) en F\'{\i}sica
At\'omica,\hfill\break\indent Molecular y Fot\'onica,
Universidad de La Laguna\hfill\break\indent C/. Astrof\'{\i}sico Francisco
S\'{a}nchez s/n, 38200 -- La Laguna, SPAIN.}
\email{{\tt jjgarmel@ull.es}}

\address{A. Quaas\hfill\break\indent
Departamento de Matem\'{a}tica, Universidad T\'ecnica Federico Santa Mar\'{\i}a
\hfill\break\indent  Casilla V-110, Avda. Espa\~na, 1680 --
Valpara\'{\i}so, CHILE.}
\email{{\tt alexander.quaas@usm.cl}}

%%%%%%%%0.Abstract

\begin{abstract}
In this work we obtain a Liouville theorem for positive, bounded solutions of the 
equation 
$$
(-\De)^s u= h(x_N)f(u) \quad \hbox{in }\mathbb{R}^{N}
$$
where $(-\De)^s$ stands for the fractional Laplacian with $s\in (0,1)$, and the 
functions $h$ and $f$ are nondecreasing. The main feature is that the function 
$h$ changes sign in $\R$, therefore the problem is sometimes termed as indefinite. 
As an application we obtain a priori bounds for positive solutions of some 
boundary value problems, which give existence of such solutions by means of 
bifurcation methods.
\end{abstract}

\maketitle

\section{Introduction and main results}
\setcounter{section}{1}
\setcounter{equation}{0}

The objective of the present paper is to obtain a Liouville theorem for a nonlocal elliptic 
equation involving the fractional Laplacian. This operator is defined for sufficiently smooth functions
by
$$
(-\Delta)^s u(x) = c(N,s) \int_{ \mathbb{R}^N} \frac{u(x)-u(y)}{|x-y|^{N+2s}} dy,
$$
where $0<s<1$, $c(N,s)$ is a normalization constant whose value will be of no importance for us and 
the integral is to be understood in the principal value sense.

During the last years there has been an increasing amount of research on equations driven by $(-\De)^s$. 
The main interest is to test whether the known features for its local counterpart $-\De$, obtained by setting $s=1$, 
remain valid for arbitrary $s\in (0,1)$. In general, this has led to adaptation of the standard 
techniques and to the search for new tools. Moreover, sometimes even stronger results can be obtained 
in the nonlocal case as our main theorem below shows.

\medskip

When it comes to Liouville theorems there is a more or less satisfactory understanding of the 
equation 
\begin{equation}\label{eq-intro-1}
(-\De)^s u = f(u) \quad \hbox{in }\R^N.
\end{equation}
The case $f(t)=t^p$ was considered in \cite{CLO1, ZCCY}, while in \cite{CLZ} some more 
general nonlinearities were analyzed. In these works, the authors obtained for arbitrary $s\in (0,1)$ 
an analogue of the previously proved results in $s=1$ (cf. \cite{GS1,CL,B,LZ}).

However, the situation is fairly different when \eqref{eq-intro-1} is replaced by a \emph{non-autonomous} 
equation. Let us mention the papers \cite{DZ}, \cite{FaW}, which deal with H\'enon equation
$$
(-\De)^s u= |x|^\alpha u^p \quad \hbox{in } \R^N,
$$
where $\al>0$ and $p>1$. As for equations which involve weights with change sign, only the paper 
\cite{CZ} is known to us. There, the authors consider the equation 
\begin{equation}\label{eq-intro-2}
(-\De)^s u= x_N u^p \quad \hbox{in } \R^N,
\end{equation}
and show that there do not exist positive, bounded solutions for any $p>1$, provided that $s\ge \frac{1}{2}$. 
This result is subsequently used to obtain a priori bounds for solutions of some related boundary 
value problems (see \eqref{bvp} below). The main technique in \cite{CZ} is the reduction of the problem to a local one 
by means of the extension problem introduced in \cite{CS3}. The notation in \eqref{eq-intro-2} is the usual 
one: for a point $x\in \R^N$ we write $x=(x',x_N)$, where $x'\in \R^{N-1}$ and $x_N\in \R$.

The local version of problems related to \eqref{eq-intro-2} has been also considered for instance in the works 
\cite{BCDN} and \cite{L}, but in our opinion perhaps the most general result in this regard 
was obtained in \cite{Du-Li}, where the problem 
$$
-\De u = h(x_N) f(u) \quad \hbox{in } \R^N
$$
is studied. Here $h$ and $f$ are nondecreasing functions which verify some additional conditions, 
the main feature being that the function $h$ is assumed to be nonpositive for $x_N<0$ and positive 
for $x_N>0$. As for the nonlinearity $f$, the natural example is $f(t)=t^p$, with $p>1$.
 
Our intention in this work is to obtain a similar Liouville theorem for the problem
\begin{equation}\label{problema}
(-\De)^s u= h(x_N)f(u) \quad \hbox{in }\mathbb{R}^{N},
\end{equation}
where both $h$ and $f$ are monotone and $h$ is allowed to change sign. We state below our 
precise hypotheses on $h$ and $f$, but it is interesting to remark that they are less 
stringent than in the case $s=1$ considered in \cite{Du-Li}.

\medskip

On the functions $h$ and $f$ we will assume the following, which will be termed altogether 
as hypotheses (H):

\smallskip

\begin{itemize}
\item[(H1)] $h\in C^\al(\R^N)$ for some $\al\in (0,1)$.
\item[(H2)] $h$ is nondecreasing in $\R$, with $h(0)=0$ and $h(t)>0$ for $t>0$. 
\item[(H3)] $\ds \lim_{t\to +\infty} h(t)=+\infty$.
\item[(H4)] $f$ is locally Lipschitz and nondecreasing in $[0,+\infty)$, with $f(0)=0$ and $f>0$ 
in $(0,+\infty)$. \
\item[(H5)] $\ds \lim_{r, t\to 0}\frac{f(r)-f(t)}{r-t}=0$.
\end{itemize}

\smallskip

\noindent Observe that condition (H2) could be stated with respect to another point different from zero, 
but this amounts only to a change of variables in \eqref{problema}. Natural examples for functions 
$h$ and $f$ are $h(t)=|t|^{\al-1} t$ for some $\al>0$ and $f(t)=t^p$ for $p>1$. The case $\al=1$ 
then leads to \eqref{eq-intro-2}. 

Let us also mention here that, if $f\in C^1$ near the origin then condition (H5) is equivalent to $f'(0)=0$. 
Nevertheless, the case $f'(0)>0$ could also be included in our main theorem  by arguing as in \cite[Pag 13]{BDGQ2}. 
However, since the main application of the Liouville theorem presented in this work is concerned with 
the existence result established in Theorem \ref{th-bvp} below, we have considered that this case could 
be omitted.

We now come to the statement of our Liouville theorem for \eqref{problema}. We will be dealing 
throughout with classical solutions, that is, functions $u\in C^{2s+\beta}(\R^N)$ for 
some $\beta\in (0,1)$, verifying \eqref{problema} at every point of $\R^N$. However, it is to be 
noted that by bootstrapping and the regularity theory developed in \cite{CS} and \cite{S}, solutions in the 
viscosity sense turn out to be classical.

\begin{teorema}\label{th-Liouville}
Assume $h$ and $f$ verify hypotheses {\rm (H)}. Then problem \eqref{problema} does not admit any 
positive, bounded solution.
\end{teorema}

\medskip

A natural application of this Liouville theorem arises when considering boundary value problems with 
indefinite weights, for instance
\begin{equation}\label{bvp}
\left\{
\begin{array}{ll}
(-\De)^s u = \la u + a(x) u^p & \hbox{in } \Om,\\
\ \ u=0 & \hbox{in }  \R^N \setminus \Om,
\end{array}
\right.
\end{equation}
where $a\in C^\al(\Omb)$ for some $\al\in (0,1)$. Here $p>1$, $\la\in \R$ is a parameter and 
$a$ is assumed to change sign in a ``controlled" way. The local case $s=1$ has been extensively 
studied, to mention a few, in \cite{AT1, AT2, ALG, BCDN, BCDN2, BZ, CL2, CL3, Du-Li} 
(see more references in \cite{Du-Li}). 

As for the fractional case $s\in (0,1)$, we refer to \cite{GS} and \cite{FL}, where variational 
techniques were used. The use of variational techniques allows for somewhat relaxed hypotheses, however 
they only give existence of positive solutions of \eqref{bvp} for positive values of $\la$. On the 
other hand, the approach which we follow here, based on a priori bounds and bifurcation theory, 
is suitable for generalization to a nonvariational setting. Indeed, the a priori bounds can be 
obtained as in \cite{BDGQ1}, while the application of bifurcation theory requires only minor technical 
adjustments (which however go beyond the scope of this work).

In the present situation, we will be assuming that $a$ verifies the structural 
conditions, termed henceforth as hypotheses (A):

\smallskip

\begin{itemize}

\item[(A1)] The set $\Gamma:=\{x\in \Omb:\ a(x)=0\}$ is a smooth manifold of dimension $N-1$ 
contained in $\Om$. 

\item[(A2)] There exist $\g>0$ and positive, continuous functions $b_1$, $b_2$ defined in $\Om$ such 
that in a neighborhood of $\Gamma$
$$
a(x) = \left\{
\begin{array}{ll}
b_1(x) d(x)^\g & x\in \Om^+\\[0.25pc]
-b_2(x) d(x)^\g & x\in \Om^-
\end{array}
\right.
$$
where $d(x):=\hbox{dist}(x,\Gamma)$, $\Om^+:=\{x\in\Om: \ a(x)>0\}$, 
$\Om^-:=\{x\in\Om: \ a(x)<0\}$.
\end{itemize}

\smallskip

Observe that hypotheses (A) imply that the set $\Gamma=\{a=0\}$ is contained in $\Om$ and has empty interior. 
Moreover, $a$ has different signs on ``both sides" of $\Gamma$. This is equivalent to saying that $\overline{\Om^+} 
\cap \overline{\Om^-}=\G$ (see Figure 1).

\smallskip

\psfrag{O}{$\Om$}
\psfrag{O+}{$\Om^+$}
\psfrag{O-}{$\Om^-$}
\psfrag{G}{$\Gamma$}

\begin{center}
	\includegraphics[width=4.5cm,height=3cm]{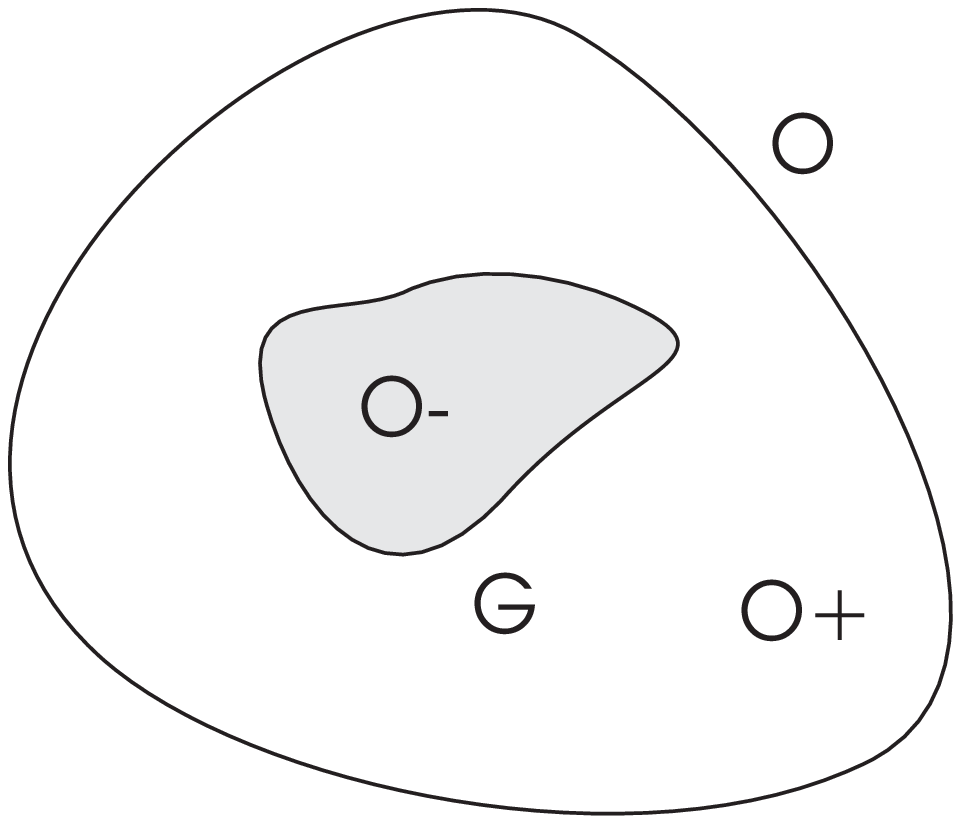}

\medskip
{\sc Figure 1.} A possible configuration for $\Om^+$

and $\Om^-$ in hypotheses (A).
\end{center}

\bigskip

With these assumptions on the weight $a$, and assuming in addition that $p$ is subcritical, 
we can obtain a priori bounds for all positive solutions of \eqref{bvp} in bounded $\la$-intervals. 
And with the aid of bifurcation theory, these a priori bounds lead to an existence result. We 
denote by $\la_1(\Om)$ the first eigenvalue of $(-\De)^s$ in $\Om$.

\begin{teorema}\label{th-bvp}
Assume $s\in (0,1)$, $p>1$ and let $a\in C^\al(\Omb)$ verify hypotheses {\rm (A)}. If $N\le 2s$ or $N>2s$ and 
$p$ is such that
$$
p<\frac{N+2s}{N-2s},
$$
then problem \eqref{bvp} admits at least a positive classical solution for every $\la <\la_1(\Om)$. 
Moreover, there exists $\La\ge \la_1(\Om)$ such that there are no such solutions 
if $\la >\La$. 
\end{teorema}

\bigskip

It is worthy of mention that in some cases one can guarantee the existence of positive solutions 
of \eqref{bvp} also for values $\la>\la_1(\Om)$. Indeed, a more precise description of the bifurcation 
diagram near the point $(\la_1(\Om),0)$ can be performed by using for instance the Crandall-Rabinowitz 
theorem (cf. \cite{CR}), but this is definitely out of the scope of this article. See also 
\cite{GS} and \cite{FL}.

\medskip

Let us briefly mention our method of proof. For the Liouville theorem we will use the moving 
planes method to establish monotonicity of the solution in the direction $x_N$. Nevertheless, 
the application of the method is by no means standard, since in 
spite of the problem being posed in $\R^N$, at some point we can deal with problems in half-spaces, 
which allows us to introduce the Green's function obtained in \cite{FW}, as was done in our 
previous work \cite{BDGQ2}. The use of the Green's function is what definitely distinguishes the 
case $s\in (0,1)$ from $s=1$, thus allowing the hypotheses to be less restrictive. 

As for the a priori bounds, we follow the approach in \cite{BDGQ1} (but see also \cite{CLL} for a 
related approach). The blow-up method introduced in \cite{GS2} is used, but we need to resort to 
the barriers introduced in \cite{BDGQ1} when the limit problem is posed in a half-space. Finally, 
Theorem \ref{th-bvp} can be achieved with an application of the global bifurcation theorem of 
Rabinowitz (cf. \cite{R}), and an analysis along the lines of the one made in \cite{DPQ}.

\medskip

The rest of the paper is organized as follows: in Section 2 we will prove our Liouville theorem. 
Section 3 is dedicated to obtaining the a priori bounds, while in Section 4 we will consider 
the question of existence of solutions.

\bigskip

\section{The Liouville Theorem}
\setcounter{section}{2}
\setcounter{equation}{0}

This section is dedicated to the proof of Theorem \ref{th-Liouville}. The main step in the proof is 
to show that any bounded, positive solution $u$ of \eqref{problema} has to be increasing in the $x_N$ 
direction. For this we will use the moving planes method as in \cite{BDGQ2} (see also \cite{FW}).

We begin by introducing some notation, which is for the most part rather standard in this context. We 
denote $x=(x',x_N)$ for points $x\in \R^N$ and for $\la\in \R$, let
$$
\begin{array}{l}
\Sigma_\la :=\{ x\in \R^N:\ x_N<\la\},\\[.5pc]
T_\la :=\{x\in \R^N:\ x_N=\la\},\\[.5pc]
x^\la:=(x',2\la-x_N) \ \hbox{(the reflection of }x \hbox{ with respect to } T_\la).
\end{array}
$$
For a positive, bounded, classical solution $u$ of problem \eqref{problema}, we 
also set
$$
\begin{array}{ll}
u_\la(x)= u(x^\la)\\[0.25pc]
w_\la(x)= u_\la(x)-u(x) 
\end{array}
\quad x\in \R^N.
$$

\smallskip

\begin{proof}[Proof of Theorem \ref{th-Liouville}]
Assume there exists a positive, bounded, classical solution $u$ of \eqref{problema}. The 
proof proceeds in two main stages: first we show that $u$ is increasing in the 
$x_N$ direction, that is, $w_\la\ge 0$ in $\Sigma_\la$ for every $\la\in \R$. Then we will prove that this 
is impossible by a simple principal eigenvalue argument. All this will be accomplished in 
a series of steps.

\bigskip

\noindent {\bf Step 1}. $w_\la \ge 0$ in $\Sigma_\la$ when $\la \le 0$.

\smallskip

\noindent By contradiction let us suppose that there exists $\la\leq 0$ such that $w_\la < 0$ somewhere in 
$\Sigma_\la$. Then we can define, for $\lambda\leq 0$, the nonempty open set
\begin{equation}\label{D}
D_\la=\{x\in \Sigma_\la:\ w_\la(x)<0\},
\end{equation}
and the function 
\begin{equation}\label{v_z}
v_\la = w_\la \chi_{D_\la}\leq 0.
\end{equation}
Observe that $x_N^\la > x_N$ when $x\in\Sigma_{\lambda}$, so that the monotonicity of both $h$ and $f$ and the nonnegativity of 
$f$ give, for $x\in D_\la$: 
\begin{equation}\label{claim1}
(-\Delta)^{s} w_\la (x) \geq h(x_N)(f(u^\la(x))-f(u(x)))\geq 0,
\end{equation}
since $h(x_N)\le 0$ for $x\in \Sigma _\la$ when $\la \le 0$.

Arguing as in Lemma 5 of \cite{BDGQ2}, we see that 
\begin{equation}\label{bbb}
(-\De)^s v_\la \ge (-\De)^s w_\la \ge 0 \quad \hbox{in }D_\la.
\end{equation} 
Since $v_\la =0$ in $\R^N\setminus D_\la$, we may apply the maximum principle for open sets contained 
in a half-space (Lemma 4 in \cite{BDGQ2}) 
to obtain that $w_\la \ge 0$ in $D_\la$, a contradiction. This contradiction shows that $D_\la$ is 
empty, so that $w_\la \ge 0$ in $\Sigma_\la$ when $\la\le 0$.

\bigskip

\noindent {\bf Step 2}. Setting 
$$
\la_*:=\sup\{ \la \in \R:\  w_\mu \ge 0  \hbox{ in } \Sigma_\mu \hbox{ for every } \mu<\la\},
$$
we have $\la_*=+\infty$.

\smallskip

\noindent Assume for a contradiction that $\la_*<+\infty$. We first 
observe that, by the definition of $\la_*$, there exists a sequence of positive numbers 
$\{\la_n\}$ such that $\la_n\downarrow \la_*\ge 0$  
and points $x_n \in \Sigma_{\la_n}$ such that $w_{\la_n} (x_n)<0$. From now on, we will 
use the notation $w_n$, $v_n$ and $D_n$ instead of $w_{\la_n}$, $v_{\la_n}$ and $D_{\la_n}$, 
where all these functions are obtained by just setting $\la=\la_n$ in the previous definitions. 
Let us define
$$
W_{n}:=D_{n}\cap \mathbb{R}^{N}_{+},
$$
where $\R^N_+=\{x\in \R^N:\ x_N>0\}$. 
We claim that $W_n\neq \emptyset$. To prove this, assume on the contrary that  
$$
D_{n}\subseteq \overline{\mathbb{R}^{N}_{-}},
$$ 
where $\R^N_-=\{x\in \R^N:\ x_N<0\}$. Arguing as in Step 1 we see that both \eqref{claim1} and \eqref{bbb} 
hold. Thus by the maximum principle in \cite{BDGQ2} we arrive at $v_n \ge 0$ in $D_n$, which is 
a contradiction. The contradiction shows that $W_n\ne \emptyset$.

To proceed further, choose points $x_n \in W_n$ such that
\begin{equation}\label{losx}
-v_n (x_n )\geq\frac{1}{2}\|v_ n\|_{L^{\infty}(W_n)}.
\end{equation}
Notice that, by definition, $0<x_{n,N}\leq \la_n$, and we may assume by passing to a subsequence that 
\begin{equation}\label{limite}
x_{n, N}\to a\in [0,\la_*].
\end{equation} 
We define next the functions
$$
\widetilde{u}_n (x) :=u(x'+x'_n,x_N),\quad x\in \R^N,
$$
which are positive solutions of problem \eqref{problema}. In addition, they verify $\| \widetilde{u}_n\|_{L^\infty(\R^N)}=
\| u \|_{L^\infty(\R^N)}$. We also set, for $x\in \R^N$,
\begin{align*}
\widetilde{w}_n (x) & :=w_n(x'+x'_n,x_N)\\
\widetilde{v}_n (x) & :=v_n(x'+x'_n,x_N).
\end{align*}
Observe that $\widetilde{v}_n= \widetilde{w}_n \chi_{\widetilde{D}_n}$, 
where $\widetilde{D}_{n}=\{x\in\Sigma_{\la_n}:\, \widetilde{w}_n(x)<0\}$.

Since $x_n \in W_n$, it is easily seen that 
\begin{equation}\label{zn}
z_n:=(0,x_{n,N})\in \widetilde{W}_{n}:=\widetilde{D}_{n}\cap \mathbb{R}^{N}_{+}.
\end{equation}
Moreover, by \eqref{losx}, we also have
\begin{equation}\label{losx-2}
-\widetilde{v}_n (z_n )\geq\frac{1}{2}\|\widetilde{v}_ n\|_{L^{\infty}(\widetilde{W}_n)},
\end{equation}
and 
\begin{equation}\label{limite-2}
z_n \to z_0=(0,a) \in \Sigma_{\la_*}\setminus \R^N_-, 
\end{equation}
owing to \eqref{limite}. 
Our next intention is to obtain an integral inequality involving the $L^\infty$ norm of 
$\widetilde{v}_n$ in $\widetilde{W}_n$, in the spirit of \cite{FW}. 
Arguing as in Lemma 5 in \cite{BDGQ2} we deduce that
$$
\begin{array}{rl}
(-\Delta)^s \widetilde{v}_n \hspace{-2mm} & \ge h(x_N)(f(\widetilde{u}_n^{\la_n}) -
f(\widetilde{u})) \chi_{\widetilde{W_n}}\\[.25pc]
& \ge a_n \chi_{\widetilde{W_n}} \widetilde{v}_n
\end{array}
\quad \hbox{in } \Sigma_{\la_n}
$$
in the viscosity sense, where, since $x_N\leq \lambda_*+1$,
\begin{equation}\label{an}
a_n(x):= h(\la_*+1) \frac{f(\widetilde{u}_n^{\la_n})-f(\widetilde{u}_n)}
{\widetilde{u}_n^{\la_n}-\widetilde{u}_n}, \quad x\in \R^N.
\end{equation}
Proceeding now as in Lemma 6 in \cite{BDGQ2}, we obtain
\begin{equation}\label{frid}
\widetilde{v}_n(x)\geq\int_{\widetilde{W}_n}{G_{n}(x,y) a_n(y) \widetilde{v}_n(y)\, dy},\quad x\in\Sigma_{\la_n}.
\end{equation}
Here $G_n(x,y)$ stands for the Green's function in the half-space $\Sigma_{\la_n}$. Notice that, by means of a 
change of variables it is easily seen that 
\begin{equation}\label{Gn}
G_n(x,y)=G_\infty^+ (x',\la_n-x_N,y',\la_n-y_N), \quad \hbox{for }x,y \in \Sigma_{\la_n},
\end{equation}
where $G_\infty^+$ stands for the Green's function in the ``standard" half-space $\R^N_+$ (cf. \cite{FW}). 
Here and in what follows we are taking the liberty of expressing all Green's functions as depending on 
two variables $(x,y)$ or four variables $(x',x_N,y',y_N)$, hoping that no confusion arises. 
Taking $x=z_n\in\Sigma_{\la_n}$ in \eqref{frid} and using \eqref{losx-2} we arrive at
$$
\frac{1}{2}\|\widetilde{v}_n\|_{L^{\infty}(\widetilde{W}_n)}\leq  
\|\widetilde{v}_n\|_{L^{\infty}(\widetilde{W}_n)}\int_{\widetilde{W}_n}{G_{n}(z_n,y) a_n(y)\, dy}.
$$
Since we are assuming that the norms $\|\widetilde{v}_n\|_{L^{\infty}(\widetilde{W}_n)}$ are nonzero, 
we deduce that 
\begin{equation}\label{arranque}
\frac{1}{2} \le \int_{\widetilde{W}_n}{G_{n}(z_n,y) a_n(y)\, dy}.
\end{equation}
Our ultimate aim is to show that this inequality is impossible.

Observe that the sequence $\{u_n\}$ is uniformly bounded and every $u_n$ is a positive solution of 
\eqref{problema}. Then, with the use of standard regularity (cf. \cite{CS,S}) and by a diagonal argument, we 
may assume by passing to a subsequence that, for every $\beta\in (0,1)$,
$$
\widetilde{u}_n\to \widetilde{u}\mbox{ in } \mathcal{C}_{\rm loc}^{2s+\beta}(\R^N),
$$
where $\widetilde{u}$ is a nonnegative solution of \eqref{problema}. 
By the strong maximum principle, we may ensure that either $\widetilde{u}>0$ or $\widetilde{u}\equiv 0$ 
in $\R^N$. We have to analyze these two cases separately. 

\medskip

\noindent {\it Case (a)}: $\widetilde{u}>0$ in $\R^N$. We claim that 
\begin{equation}\label{interrogante}
\widetilde{w}_{\la_*}:=\widetilde{u}^{\la_*}-\widetilde{u}>0,\quad \hbox{in }\Sigma_{\la_*}.
\end{equation}
Indeed, it is clear that $\widetilde{w}_{\la_*}\ge 0$ in $\Sigma_{\la_*}$. So, suppose that there 
exists $x_0\in\Sigma_{\lambda_*}$ such that $\widetilde{w}_{\lambda_*}(x_0)=0$. Then, 
as in \eqref{claim1}, 
$$
(-\De)^s \widetilde{w}_{\la_*} (x_0) \geq h(x_{0,N})(f(\widetilde{u}^{\la_*}(x_0))-f(\widetilde{u}(x_0)))=0,
$$
where $x_{0,N}$ is the last component of the point $x_0$. On the other hand, 
using the fact that $\widetilde{w}_{\la_*}$ is antisymmetric, it follows that
\begin{align*}
0 & \le  (-\De)^s \widetilde{w}_{\la_*} (x_0)  = -\left(  \int_{\Sigma_{\la_*}} + \int_{\Sigma_{\la_*}^c} \right) 
\frac{\widetilde{w}_{\la_*}(y)}{|x_0-y|^{N+2s}} dy\\
& = -\int_{\Sigma_{\la_*}} \widetilde{w}_{\la_*}(y)\left(\frac{1}{|x_0-y|^{N+2s}} - \frac{1}{|x_0-y^{\la_*}|^{N+2s}}\right)\, dy,
\end{align*}
where we have made the change of variables $y \to y^{\la_*}$ in the integral taken in $\Sigma_{\la_*}^c$. 
Now, for $y\in \Sigma_{\la_*}$ we always have $|x_0-y| \le |x_0-y^{\la_*}|$, so that the integrand above is 
nonnegative, and we deduce that $\widetilde{w}_{\la_*}=0$ in $\R^N$ (that is, $\widetilde{u}$ is symmetric with respect to the 
hyperplane $T_{\la_*}$). This is impossible since, taking $\widetilde{x} \in \Sigma_{\la_*}$ with 
$\widetilde{x}_N<0 < 2\la_* - \widetilde{x}_N$ and $h(2\la_*-\widetilde{x}_N)>0$ we obtain 
$$
0=(-\De)^s \widetilde{w}_{\la_*} (\widetilde{x}) = f(\widetilde{u}(\widetilde{x}))(h(2\la_*-\widetilde{x}_N)-h(\widetilde{x}_N))>0,
$$
which is a contradiction. This contradiction shows \eqref{interrogante}. Notice that, for the problem at hand, 
to prove that there cannot exist symmetric solutions with respect to hyperplanes $T_{\lambda}$, 
the presence of the function $h$ is crucial, contrary to what happens in the case of a half-space with zero Dirichlet condition, 
treated in \cite{BDGQ2}.

\medskip

Next we observe that by our choice of $z_n$, it follows that $\widetilde{w}^{\la_*}(z_0)=0$. 
Since we have just shown the positivity of $\widetilde{w}^{\la_*}$ in $\Sigma_{\la_*}$,  
we have $z_0\in T_{\la_*}$, that is, $a=\la_*$ in \eqref{limite-2}. Let us show that this contradicts 
\eqref{arranque}. Indeed, using that the coefficients $a_n$ are uniformly bounded by the 
Lipschitz constant $L$ of $f$, and $\widetilde{W}_n\subset \Sigma_{\la_n}\cap \R^N_+$, we get 
from \eqref{arranque} that:
\begin{equation}\label{eq1}
\frac{1}{2} \le L \int_{\Sigma_{\la_n}\cap \R^N_+} G_n(z_n,y) dy.
\end{equation}
Taking into account the characterization \eqref{Gn} of $G_n$ and performing a change of variables 
in \eqref{eq1}, it follows that
\begin{eqnarray}
\frac{1}{2} & \le &\int _{\Sigma_{\la_n}\cap \R^N_+}  G^+_\infty(0,\la_n-z_{n,N},y) dy\nonumber\\
&\le& \int _{\Sigma_{\la_*+1}\cap \R^N_+}  G^+_\infty (0,\la_n-z_{n,N},y) dy.\label{eq2}
\end{eqnarray}
However, using part (c) of Lemma 7 in \cite{BDGQ2}, we deduce that the last integral converges to 
zero, since $\la_n-z_{n,N}\to 0$. This contradiction rules out the case $\widetilde{u}>0$ in 
$\R^N$.

\medskip

\noindent {\it Case (b)}: $\widetilde{u} \equiv 0$ in $\R^N$. 
Thanks to our hypotheses on $f$ and using that $u_n\to 0$ uniformly on compact sets of $\R^N$, 
we deduce that $a_n\to 0$ as $n\to\infty$, uniformly on compact sets of $\R^N$. 
We will prove that this entails the convergence of the right hand side of \eqref{arranque} to zero 
when $n\to\infty$, obtaining the same type of contradiction as before.

In fact, using \eqref{arranque} and \eqref{Gn}, we see that
\begin{align}\label{op3}
\frac{1}{2} & \le \int_{\Sigma_{\la_n}\cap \R^N_+} G_n(z_n,y) a_n(y) dy \nonumber\\ 
&= \int_{\Sigma_{\la_n}\cap \R^N_+}  G_\infty^+(0,\la_n-z_{n,N},y',\la_n-y_N) a_n (y',\la_n-y_N) dy \nonumber\\
&\le \| a_n\|_{L^\infty (B_R^+)} \int_{\Sigma_{\la_*+1}\cap B_R^+}  G_\infty^+(0,\la_n-z_{n,N},y',\la_n-y_N) dy\\
& +L \int_{\Sigma_{\la_*+1}\cap (B_R^+)^c}  G_\infty^+(0,\la_n-z_{n,N},y',\la_n-y_N) dy, \nonumber
\end{align}
where $B_R^+=B_R\cap \R^N_+$. A minor variation of parts (a) and (b) in Lemma 7 of \cite{BDGQ2} gives
$$
\lim_{R\to +\infty} \int_{\Sigma_{\la_*+1}\cap (B_R^+)^c}  G_\infty^+(0,\la_n-z_{n,N},y',\la_n-y_N) dy=0,
$$
uniformly in $n\in \N$, while 
$$
\int_{\Sigma_{\la_*+1}\cap B_R^+}  G_\infty^+(0,\la_n-z_{n,N},y',\la_n-y_N) dy\le C
$$
for fixed $R$. Thus we can fix $R>0$ such that the last term in \eqref{op3} is less than $\frac{1}{4}$, 
say, to get
$$
\frac{1}{4} \le C \| a_n\|_{L^\infty (B_R^+)},
$$
which is a contradiction since $a_n\to 0$ uniformly on compact sets.

\medskip

\noindent {\bf Step 3}. Completion of the proof. 

\smallskip

\noindent By Step 2, we deduce that $u$ is nondecreasing 
in the $x_N$ direction. Next, for every $k\in \N$, let $B_k=B_1(ke_N)$ be the ball of center 
$ke_N$ and radius $1$. By the monotonicity of $u$ shown above, we obtain, for every $k\ge 1$:
$$
u(x) \ge u_1=\min_{B_1} u>0, \quad x\in B_k.
$$
Setting $m_0=\inf_{u_1 \le t \le \|u\|_{L^\infty(\R^N)}} \frac{f(t)}{t}>0$, 
we see that 
$$
(-\De)^s u \ge h(k-1) m_0 u \quad \hbox{in }B_k.
$$ 
According to well-known properties of the principal eigenvalue (see for instance Theorem 1.1 of \cite {QSX}), 
we deduce that 
$$
h(k-1) m_0 \le \la_1 (B_k)=\la_1(B_1),
$$
and we arrive at a contradiction by letting $k\to +\infty$, since we are assuming that 
$h$ goes to infinity at infinity. The proof is concluded.
\end{proof}

\bigskip

\section{A priori bounds}
\setcounter{section}{3}
\setcounter{equation}{0}

In this section we will show that all solutions of \eqref{bvp} are a priori bounded, provided that 
$p$ is subcritical and $a$ verifies hypotheses (A). The technique is the standard one introduced in 
\cite{GS2}, with the adaptations to the nonlocal setting provided by \cite{BDGQ1}. It relies in the 
Liouville theorems obtained in \cite{ZCCY} and \cite{CLO1}, the monotonicity of solutions in half-spaces proved in 
\cite{BDGQ2} and our new Theorem \ref{th-Liouville}.

\begin{teorema}\label{th-apriori}
Assume $s\in (0,1)$, $p>1$ and let $a\in C^\al(\Omb)$ verify hypotheses {\rm (A)}. If $N\le 2s$ or $N>2s$ and 
$p$ is such that
$$
p<\frac{N+2s}{N-2s},
$$
then for every $\la_1,\la_2\in \R$ with $\la_2 > \la_1$ there exists $M=M(\la_1,\la_2)$ such that 
$$
\| u\|_{L^\infty(\Omega)} \le M
$$
for every positive, classical solution $u$ of \eqref{bvp} with $\la\in [\la_1,\la_2]$.
\end{teorema}

\begin{proof}
Assume on the contrary the existence of an interval $[\la_1,\la_2]$ and sequences 
$\{\la_k\}\subset [\la_1,\la_2]$ and $u_k$ such that $u_k$ is a positive, classical 
solution of \eqref{bvp} with $\la=\la_k$ and 
$$
M_k:= \| u_k \|_{L^\infty(\Omega)}  \to +\infty.
$$
For every $k$, take a point $x_k \in \Om$ where $u_k$ achieves its maximum. We may assume 
$x_k\to x_0\in \Omb$. Now, three cases are possible:

\begin{itemize}

\item[(a)] $x_0\in \Om\setminus \Gamma$;

\smallskip
\item[(b)] $x_0 \in \Gamma$;

\smallskip
\item[(c)] $x_0\in \p\Om$.

\end{itemize}

The cases (a) and (c) are to some extent standard, and only the remaining case (b) deserves special 
attention. Let us see that we reach a contradiction assuming each one of them.

\medskip

\noindent (a) Let $\mu_k=M_k^{-\frac{p-1}{2s}}\to 0$ and introduce the functions
$$ 
v_k (y) = \frac{ u_k (x_k +\mu_k y)}{M_k}, \quad  y\in \Om_k,
$$
where 
\begin{equation}\label{def-omega-k}
\Om_k:= \{ y\in \R^N: \ x_k +\mu_k y \in \Om \}.
\end{equation}
It can be easily seen that $\Om_k\to \R^N$ as $k\to +\infty$. On the other hand, 
it is clear that $v_k$ verifies $0<v_k\le 1$ and $v_k(0)=1$. Moreover, a short calculation 
shows that $v_k$ is a solution of the problem 
$$
\left\{
\begin{array}{ll}
(-\De)^s v_k = \la_k \mu_k ^{2s} v_k + a_k (y) v_k^p & \hbox{in } \Om_k\\
\ \ v_k=0 & \hbox{in }  \R^N \setminus \Om_k,
\end{array}
\right.
$$
where $a_k(y)= a(x_k+\mu_k y)$, $y\in \Om_k$. Thus we may use standard regularity (cf. \cite{CS,S}) 
to obtain that, through a subsequence, $v_k\to v$ uniformly in compact sets of $\R^N$, where 
$v$ is a nonnegative, bounded, viscosity solution of 
$$
(-\De)^s v = a(x_0) v^p \quad \hbox{in } \R^N.
$$
By the strong maximum principle and regularity theory we actually have that $v$ is a 
positive, classical solution. 

By our hypotheses in this case, we know that $a(x_0)\ne 0$. If $a(x_0)<0$, we observe 
that $v(0)=1$ while $v\le 1$. Thus $v$ attains a global maximum at $y=0$ and 
$(-\De)^s v(0)=a(x_0)<0$, which is impossible. If, on the contrary, $a(x_0)>0$, 
we reach a contradiction when $N\le 2s$ by Theorem 1.2 in \cite{FQ}, and when $N>2s$ and $p$ is subcritical  
by Theorem 4 in \cite{ZCCY} (see also \cite{CLO1}). 

\medskip

\noindent (b) We may assume with no loss of generality that the outward normal  
to $\p \Om^+$ at $x_0$ is $\nu(x_0)=-e_N$. Let $d_k=d(x_k)$ and recall that we 
are denoting $d(x)=\hbox{dist}(x,\Gamma)$. Define
$$
\eta_k=M_k^{-\frac{p-1}{2s+\gamma}}\to 0
$$
and 
$$ 
\widetilde{v}_k (y) = \frac{ u_k (x_k +\eta_k y)}{M_k}, \quad  y\in \widetilde{\Om}_k,
$$
where $\widetilde{\Om}_k:= \{ y\in \R^N: \ x_k +\eta_k y \in \Om \}$. Observe that in 
the present situation $\widetilde{\Om}_k\to \R^N$ as $k\to +\infty$. It is not hard to 
see that $\widetilde{v}_k$ verifies the equation
\begin{equation}\label{eq-limite-2}
(-\De)^s \widetilde{v}_k = \la_k \eta_k^{2s} \widetilde{v}_k + \widetilde{a}_k (y) \widetilde{v}_k^p \quad 
\hbox{in } \widetilde{\Om}_k,
\end{equation}
where $\tilde a_k(y)= \eta_k^{-\gamma} a(x_k+\eta_k y)$. It is also immediate that 
$0<\widetilde{v_k}\le 1$ and $\widetilde{v}_k(0)=1$.

Observe next that, by the smoothness assumption on $\Gamma$, we can write
$$
d(x_k+\eta_k y)= d_k +\eta_k y\cdot \nu(\xi_k)+ o (\eta_k),
$$
where $\xi_k$ is the projection of $x_k$ onto $\Gamma$. Therefore, by hypotheses (A),
$$
\widetilde{a}_k(y)=
\left\{
\begin{array}{ll}
b_1(x_k+\eta_k y) \left| \frac{d_k}{\eta_k} + \nu(\xi_k)\cdot y + o(1)\right|^{\g}, & 
\hbox{if } x_k+\eta_k y\in \Om^+\\
-b_2(x_k+\eta_k y) \left| \frac{d_k}{\eta_k} + \nu(\xi_k)\cdot y + o(1)\right|^{\g}, & 
\hbox{if } x_k+\eta_k y\in \Om^-.
\end{array}
\right.
$$
There are two further possibilities to consider:

\begin{itemize}

\item[(b1)] Passing to a subsequence, $d_k/\eta_k\to d\ge 0$. Then 
$$
\widetilde{a}_k(y) \to b_1(x_0) (y_N-d)_+^\g - b_2(x_0) (y_N-d)_-^\g,
$$
and we are denoting, for real $t$, $t_+=\max\{t,0\}$, $t_-=\max\{-t,0\}$. 
Since the sequence $\widetilde{v}_k$ is bounded, we may pass to the limit as before to obtain 
that $\widetilde{v}_k\to \widetilde{v}$ locally uniformly in $\R^N$, where $\widetilde{v}$ is 
nonnegative, bounded and verifies $\widetilde{v}(0)=1$. Moreover, passing to the limit in 
\eqref{eq-limite-2} we see that 
\begin{equation}\label{eq-otra-contra}
(-\De)^s \widetilde{v}= h(y_N) \widetilde{v}^p \quad \hbox{in } \R^N,
\end{equation}
in the viscosity sense, 
with $h(t) = b_1(x_0) (t-d)_+^\g -  b_2(x_0) (t-d)_-^\g$, $t\in \R$. With a further translation 
in the $y_N$ direction we may suppose $d=0$. Moreover, by regularity we actually 
find that $\widetilde{v}$ is a classical solution of \eqref{eq-otra-contra}, which contradicts 
Theorem \ref{th-Liouville}.

\medskip

\item[(b2)] Passing to a subsequence, $d_k/\eta_k\to +\infty$. This assumption implies 
that for large $k$ all points $x_k$ remain inside $\Om^+$ or $\Om^-$. Assume first 
that they all lie in $\Om^+$. Let $\beta_k= (\eta_k/d_k)^\frac{\g}{2s}\to 0$ and introduce 
the functions
$$
\widetilde{w}_k (y)= \widetilde{v}_k(\beta_k y),
$$
which are easily seen to verify 
$$
(-\De)^s \widetilde{w}_k = \la_k (\beta_k\eta_k)^{2s} \widetilde{w}_k + \beta_k^{2s} \widetilde{a}_k (\beta_k y) 
\widetilde{w}_k^p 
$$
in $\{y\in \R^N: \ \beta_k y \in \widetilde{\Om}_k\}$. Moreover:
$$
\beta_k^{2s} \widetilde{a}_k (\beta_k y) =  b_1(x_k+\beta_k \eta_k y) 
(1+\beta_k ^{1+\frac{2s}{\g}} \nu(\xi_k)\cdot y + o(\beta_k^\frac{2s}{\g}))^\g.
$$
Thus we see that $\widetilde{w}_k\to \widetilde{w}$, which is a bounded, nontrivial 
solution of 
$$
(-\De)^s \widetilde{w} = b_1(x_0) \widetilde{w}^p \quad \hbox{in } \R^N,
$$
a contradiction to our hypotheses as in case (a), since $b_1(x_0)>0$. When the points $x_k$ lie in $\Om^-$ 
we argue exactly in the same way and we obtain a solution of the same equation with $b_1(x_0)$ replaced by 
$-b_2(x_0)$. The contradiction follows also as in case (a).
\end{itemize}

\medskip

\noindent (c) Assume again with no loss of generality $\nu(x_0)=-e_N$, where this time 
$\nu$ stands for the outward unit normal to $\p \Om$. Denote also $d_\Om(x)={\rm dist}(x,\p\Om)$. 
There are two cases to consider: by passing to a further subsequence, either $d_\Om(x_k) \mu_k^{-1}\to 
+\infty$ or $d_\Om(x_k) \mu_k^{-1}\to d\ge 0$, where $\mu_k=M_k^{-\frac{p-1}{2s}}\to 0$, as in case (a). 
In the former one we argue exactly as in case (a) to reach a contradiction. It is to be noted that 
the set $\Om_k$ given in \eqref{def-omega-k} verifies $\Om_k\to \R^N$ with this assumption. 

In the latter we reason as in \cite{BDGQ1}. Consider the projections $\tau_k$ of $x_k$ 
onto $\p \Om$, and introduce the functions
$$
w_k(y)= \frac{u_k(\tau_k+\mu_k y)}{M_k}, \quad y \in D_k,
$$
where
$$
D_k=\{y\in \R^N:\ \tau_k+\mu_k y \in \Om\}.
$$
It is immediate that $0\in \p D_k$ and $D_k\to \R^N_+$ as $k\to +\infty$. Moreover, $w_k$ 
solves the equation
\begin{equation}\label{eq-rescale-1}
\left\{
\begin{array}{ll}
(-\De)^s w_k = \la_k \mu_k ^{2s} w_k + \bar{a}_k (y) w_k^p & \hbox{in } D_k\\
\ \ w_k=0 & \hbox{in }  \R^N \setminus D_k,
\end{array}
\right.
\end{equation}
where now $\bar{a}_k(y)=a(\tau_k+\mu_k y)$.

Next consider the point $y_k = (x_k-\tau_k)/\mu_k\in D_k$. It is clear that $w_k(y_k)=1$, and 
that $|y_k|= d_\Om(x_k)\mu_k^{-1}\to d$. We claim that $d>0$. 

Indeed, by Lemma 6 in \cite{BDGQ1}, we can choose $\theta\in (s,2s)$ and $C>0$, 
$\de>0$ such that 
$$
w_k(y) \le C d_k(y)^{2s-\theta}, \quad \hbox{when } d_k(y)\le \de,
$$
where $d_k(y):=\hbox{dist}(y,\p D_k)$. Taking into account that $|y_k|\ge d_k(y_k)$ because 
$0 \in \p D_k$, we obtain
$$
1=w_k(y_k) \le C |y_k|^{2s-\theta},
$$
provided that $d_k(y_k)\le \de$, which shows that $|y_k|$ is bounded from below. This entails 
$d>0$. Passing to another subsequence, we have $y_k\to y_0$ with $|y_0|=d>0$, therefore 
$y_0\in \R^N_+$.

Arguing as in part (a) we obtain that $w_k\to w$ locally uniformly on compact sets of $\R^N_+$, 
where $w$ verifies $0\le w\le 1$, $w(y_0)=1$ and $w(y) \le Cy_N^{2s-\theta}$ for $y_N\le \de$. 
Thus $w$ is continuous in $\R^N$ and vanishes in $\R^N\setminus\R^N_+$. Then, we can pass 
to the limit in \eqref{eq-rescale-1} to obtain that $w$ is a nonnegative, bounded, viscosity solution of the problem 
$$
\left\{
\begin{array}{ll}
(-\De)^s w= a(x_0) w^p & \hbox{in }\R^N_+\\
\ \ w=0 & \hbox{in } \R^N \setminus \R^N_+.
\end{array}
\right.
$$
By regularity theory and the maximum principle it actually follows that $w\in C^\infty(\R^N_+)$ 
is positive in $\R^N_+$. Moreover, $w$ attains a global maximum at $y_0$, thus $\nabla w(y_0)=0$. 
This contradicts Theorem 1 in \cite{BDGQ2}, which ensures that $\frac{\p w}{\p y_N}>0$ in $\R^N_+$.

\medskip

Summing up, our assumption $M_k\to +\infty$ leads to a contradiction in all cases, and this 
concludes the proof of the theorem.
\end{proof}

\bigskip

\section{Proof of Theorem \ref{th-bvp}}
\setcounter{section}{4}
\setcounter{equation}{0}

In this section we will give the proof of our existence result, Theorem \ref{th-bvp}. 
The main tool is bifurcation theory, with the use of the a priori bounds given in 
Theorem \ref{th-apriori}.

Instead of working with \eqref{bvp}, it is more convenient to deal with its odd extension, namely
\begin{equation}\label{bvp-impar}
\left\{
\begin{array}{ll}
(-\De)^s u = \la u + a(x) |u|^{p-1} u & \hbox{in } \Om\\
\ \ u=0 & \hbox{in }  \R^N \setminus \Om.
\end{array}
\right.
\end{equation}
From now on, by a solution of \eqref{bvp-impar} we will mean a pair $(\la,u)$. 
Also, it will slightly simplify our proofs to consider solutions in the viscosity sense. 
As we have already remarked, viscosity solutions are indeed classical, so this will not 
suppose any loss in generality. Thus our natural space will be $\R\times C(\Omb)$.

Problem \eqref{bvp-impar} always possesses the branch of trivial solutions $\{(\la,0): \la \in \R\}$, 
but we are only interested in nontrivial solutions. Our purpose is to obtain positive 
solutions which bifurcate from the branch of trivial solutions at the value $(\la_1(\Om),0)$, 
where $\la_1(\Om)$ stands for the first eigenvalue of $(-\De)^s$ in $\Om$.  

To make this more precise, we recall that a continuum $\mathcal{C}\subset \R\times C(\Omb)$ is a closed 
connected set. We will say that $\mathcal{C}$ is a continuum of positive solutions which bifurcates from 
$(\la_1(\Om),0)$ if $(\la_1(\Om),0)\in \mathcal{C}$ 
and $\mathcal{C}\setminus \{(\la_1(\Om),0)\}$ consists of positive solutions only.
The next result is a consequence of the celebrated global bifurcation theorem of 
Rabinowitz.

\begin{lema}\label{lema-rabi}
Assume $p>1$ and $a\in C^\al(\Omb)$ for some $\al\in (0,1)$. Then there exists an unbounded continuum 
$\mathcal{C}_0 \subset \R\times C(\Omb)$ of positive solutions of \eqref{bvp-impar} 
bifurcating from $(\la_1(\Om),0)$.
\end{lema}

\begin{proof}
For $h\in C(\Omb)$, consider the boundary value problem 
\begin{equation}\label{eq-lineal-aux} 
\left\{
\begin{array}{ll}
(-\De)^s v = h(x) & \hbox{in } \Om\\
\ \ v=0 & \hbox{in }  \R^N \setminus \Om.
\end{array}
\right.
\end{equation}
It is well-known that there exists a unique viscosity solution $v \in C(\Omb)$ of \eqref{eq-lineal-aux}. 
By Proposition 1.2 in \cite{ROS} we also have $v\in C^s(\Omb)$ and 
$$
\| v\|_{C^s(\Omb)} \le C \|h \|_{L^\infty(\Om)},
$$
for some positive $C$ independent of $h$. In this way, 
setting $v=Kh$, we define a compact, linear operator $K:C(\Omb)\to C (\Omb)$. 

Problem \eqref{bvp-impar} is equivalent to the fixed point equation
\begin{equation}\label{fpe}
u= \la K u + K (a |u|^{p-1}u)
\end{equation}
in $C(\Omb)$ (with a slight abuse of notation, we are still denoting by $a |u|^{p-1}u$ the 
Nemytskii operator of the function $a(x) |u|^{p-1} u$ defined from $C(\Omb)$ to $C(\Omb)$). Denote
$$
\mathcal{S}= \{ (\la,u)\in \R\times C(\Omb):\ u \hbox{ is a nontrivial solution of \eqref{fpe}}\}.
$$
We can apply Theorem 1.3 in \cite{R} to deduce that $\mathcal{S}$ contains a continuum $\mathcal{C}$ 
such that $(\la_1(\Om),0)\in \mathcal{C}$ and it is either unbounded or contains a point 
$(\mu,0)$ with $\mu\ne \la_1(\Om)$.

\smallskip

Next we argue as in \cite{DPQ}. 
Denote by $\mathcal{P}$ the set of functions in $C(\Omb)$ which do not change sign (observe that this is 
a closed set). We claim that:
\begin{equation}\label{eq-ult-claim}
\mathcal{C} \subset \R\times \mathcal{P}.
\end{equation}
We begin by proving the existence of a small $\e>0$ such that all solutions $(\la,u)\in \mathcal{C}$ with 
$\la\in B_\e(\la_1(\Om))\subset \R$ and $u \in B_\e(0)\subset C(\Omb)$ belong to $\R\times \mathcal{P}$. 
Indeed, suppose on the contrary that there exist sequences $\la_n \to \la_1(\Om)$, $u_n\to 0$ such that $u_n$ 
is a changing sign solution of \eqref{fpe} with $\la=\la_n$. Let 
$$
v_n= \frac{u_n}{\| u_n \|_{L^\infty(\Om)}}.
$$
It is easily seen that
\begin{equation}\label{fpe-2}
v_n= \la_n K v_n + \|u_n\|_{L^\infty(\Om)}^{p-1} K (a |v_n|^{p-1}v_n),
\end{equation}
and hence by the compactness of $K$ we see that there exists $v\in C(\Omb)$ such that $v_n\to v$ uniformly 
in $\Omb$ and $\| v\|_{L^\infty(\Om)}=1$. Passing to the limit in \eqref{fpe-2}, it is clear that $v$ verifies 
$v= \la_1(\Om) Kv$, that is, $v$ is an eigenfunction of $(-\De)^s$ associated to $\la_1(\Om)$. Hence 
we may assume with no loss of generality that $v\ge 0$ in $\Om$ and the strong maximum principle and Hopf's 
principle imply then that $v>0$ in $\Om$ and there exists $c>0$ such that 
\begin{equation}\label{posi}
v(x) \ge c d_\Om(x)^s \quad \hbox{for } x\in \Om,
\end{equation}
where $d_\Om(x)=\hbox{dist}(x,\p \Om)$ (cf. Proposition 2.7 in \cite{CRS}). Moreover, it is easily 
seen that 
\begin{equation}\label{eq-0}
\left\{
\begin{array}{ll}
(-\De)^s (v_n-v) = \la_n v_n - \la_1(\Om) v + a(x) |u_n|^{p-1} v_n & \hbox{in } \Om\\
\ \ v_n-v=0 & \hbox{in }  \R^N \setminus \Om.
\end{array}
\right.
\end{equation}
Since the right-hand side of the equation in \eqref{eq-0} converges uniformly to zero in $\Omb$, we may employ 
Theorem 1.2 in \cite{ROS} to deduce that 
$$
\left\| \frac{v_n-v}{d_\Om^s} \right\|_{C^\al(\Omb)} \to 0
$$
for some $\al\in (0,1)$. It then follows from \eqref{posi} that $v_n>0$ in $\Om$ if $n$ is large, against the assumption. 
This shows that $\mathcal{C}\cap (B_\e(\la_1(\Om))\times B_\e(0)) \subset \R\times \mathcal{P}$ for 
some small $\e>0$.

\smallskip

To prove that $\mathcal{C}\subset \R\times \mathcal{P}$, it is enough to show that 
$\mathcal{C}\cap (\R\times\mathcal{P})\cap (\R\times \overline{\mathcal{P}^c})=\emptyset$. If this were proved, we would 
have $\mathcal{C}= (\mathcal{C}\cap (\R\times \mathcal{P})) \cup 
(\mathcal{C}\cap (\R\times \overline{\mathcal{P}^c}))$, where 
these sets are disjoint and closed. Since we have shown that the first one is nonempty, the connectedness of $\mathcal{C}$ 
implies that the second one is empty, showing that $\mathcal{C}\subset \R\times \mathcal{P}$.

Thus let $(\la_0,u_0)\in \mathcal{C}\cap (\R\times\mathcal{P})\cap (\R\times \overline{\mathcal{P}^c})$. By the 
first part of the proof, we have $(\la_0,u_0)\ne (\la_1(\Om),0)$. Since $u_0\in \mathcal{P}$, we may 
assume without loss of generality that $u_0\ge 0$ in $\Omb$. By the strong maximum principle, either $u_0\equiv 0$ 
or $u_0>0$ in $\Om$. The first possibility leads, reasoning as above, to $\la_0=\la_1(\Om)$, which is impossible. 
Thus the second possibility holds and Hopf's principle gives in addition $u_0(x)\ge c d_\Om(x)^s$ in $\Om$ for some 
$c>0$. 

On the other hand, since $(\la_0,u_0)\in \R \times \overline{\mathcal{P}^c}$, there exists a sequence $(\la_n,u_n)
\subset \mathcal{C}$ such that $\la_n\to \la_0$, $u_n\to u_0$ and $u_n$ changes sign. We can argue as before to 
obtain that actually 
$$
\left\| \frac{u_n-u_0}{d_\Om^s}\right\|_{C^\al(\Omb)}\to 0
$$
as $n\to +\infty$, which implies that $u_n>0$ in $\Om$ for large $n$, a contradiction. The contradiction shows 
that $\mathcal{C}\cap (\R\times\mathcal{P})\cap (\R\times \overline{\mathcal{P}^c})=\emptyset$, thus establishing 
\eqref{eq-ult-claim}.

As a consequence of \eqref{eq-ult-claim}, we obtain that $\mathcal{C}$ is unbounded. Otherwise, we would 
have $(\mu,0)\in \mathcal{C}$ for some $\mu\ne \la_1(\Om)$. It is then seen much as before that $\mu$ 
is an eigenvalue of $(-\De)^s$ associated to a one-signed eigenfunction, which is impossible since 
$\mu\ne \la_1(\Om)$. Thus $\mathcal{C}$ is unbounded in $\R\times C(\Omb)$.

Finally, let $\mathcal{C}^\pm=\{(\la,u)\in \mathcal{C}:\ \pm u>0 \hbox{ in }\Om\}$. It is clear that 
$\mathcal{C}^+$ and $\mathcal{C}^-$ are disjoint, connected sets and $\mathcal{C}=\mathcal{C}^+\cup \{(\la_1(\Om),0\} 
\cup \mathcal{C}^-$. Moreover, one of them has to be unbounded. If $\mathcal{C}^+$ is unbounded, 
we just set $ \mathcal{C}_0 = \mathcal{C}^+\cup \{(\la_1(\Om),0)\}$. Otherwise, we take $\mathcal{C}_0 = \{(\la,-u): 
\ (\la,u)\in \mathcal{C}^-\}\cup \{(\la_1(\Om),0)\}$. In either case, $\mathcal{C}_0$ has the desired properties. This concludes the proof.
\end{proof}

\bigskip

Now we can proceed to the proof of Theorem \ref{th-bvp}.

\medskip

\begin{proof}[Proof of Theorem \ref{th-bvp}]
We begin by showing the nonexistence of positive solutions of \eqref{bvp} when 
$\la$ is large. This easily follows by noticing that, if $u$ is a positive solution of 
\eqref{bvp}, then
$$
(-\De)^s u \ge \la u \quad \hbox{in } \Om^+.
$$
It is well known, since $u>0$ in $\Om^+$ and $u\ge 0$ in $\R^N$, that this implies 
$$
\la <\la_1(\Om^+).
$$
Thus we may define
$$
\La= \sup\{ \la\in \R:\ \hbox{there exists a positive solution of } \eqref{bvp}\}.
$$
By definition there are no positive solutions of \eqref{bvp} when $\la \ge \La$. 

Next let us show that there exists a positive solution of \eqref{bvp} for every $\la<\la_1(\Om)$. 
By Lemma \ref{lema-rabi}, there exists an unbounded continuum of positive solutions $\mathcal{C}_0$ of 
\eqref{bvp} bifurcating from $(\la_1(\Om),0)$. Let 
$$
\mu= \sup\{ \la \in \R: (\la,u)\in \mathcal{C}_0 \hbox{ for some }u\}.
$$
Since $\mathcal{C}_0$ bifurcates from $(\la_1(\Om),0)$, it is clear that $\mu\ge \la_1(\Om)$. Now, 
we claim that there exists a positive solution of 
\eqref{bvp} for every $\la<\la_1(\Om)$, which will conclude the proof of the theorem. It is here 
where our a priori bounds are handy. 

Indeed, assume that for some $\la_0<\la_1(\Om)$ problem 
\eqref{bvp} does not admit any such solution. 
%We claim that there exists $\de>0$ such that 
%there are no positive solutions of \eqref{bvp} if $\la_0-\de \le \la \le \la_0$. Indeed, if this 
%were not true there would exist sequences of numbers $\la_n\to \la_0$ and positive solutions $u_n$ 
%of \eqref{bvp} with $\la=\la_n$.  By Theorem \ref{th-apriori}, there 
%exists $M>0$ such that $\| u_n \|_{L^\infty(\Om)} \le M$. Passing to the limit with the aid 
%of standard regularity we see that $u_n\to u$ uniformly in $\Omb$, where $u$ is a viscosity solution 
%--hence classical-- of \eqref{bvp} with $\la=\la_0$. By the strong maximum principle we have $u>0$ 
%in $\Om$ or $u\equiv 0$ in $\Om$. The second case is impossible since $\la_0<\la_1(\Om)$, hence 
%$u$ is a positive solution of \eqref{bvp} with $\la=\la_0$, a contradiction. 
Applying again Theorem \ref{th-apriori} we deduce the existence of $M_0$ such that 
every positive solution of \eqref{bvp} with $\la_0\le \la \le \mu$ verifies $\|u\|_{L^\infty(\Om)} 
\le M_0$. Since $\mathcal{C}_0$ is connected, it follows that
$$
\mathcal{C}_0 \subset [\la_0,\mu] \times B_{M_0}(0),
$$
which is impossible, since $\mathcal{C}_0$ is unbounded. This shows the claim.
\end{proof}

\bigskip

\noindent {\bf Acknowledgements.} 
	All authors were partially supported by Ministerio de Eco\-no\-m\'ia y Competitividad under grant MTM2014-52822-P (Spain). 
	B. B. was partially supported by a MEC-Juan de la Cierva postdoctoral fellowship number  FJCI-2014-20504 (Spain) and 
	Fondecyt Grant No. 1151180 and	Programa Basal, CMM. U. de Chile. 
	L. D. P. was partially supported by PICT2012 0153 from ANPCyT (Argentina).
	A. Q. was partially supported by Fondecyt Grant No. 1151180 Programa Basal, 
	CMM. U. de Chile and Millennium Nucleus Center for Analysis of PDE NC130017.

\end{document}